\documentclass[a4paper,reqno]{amsart}
\usepackage{amssymb}
\usepackage{latexsym}
\usepackage{amsmath}
\usepackage{euscript}
\usepackage{graphics,color}
\usepackage[all]{xy}
\usepackage[margin=3cm]{geometry}
\usepackage{MnSymbol} 
\usepackage{tikz}
\usepackage{tikz}
\usetikzlibrary{positioning, arrows.meta}

\usepackage[normalem]{ulem}

\newcommand{\be}{\begin{equation}}
\newcommand{\ee}{\end{equation}}
\newcommand{\ba}{\begin{eqnarray}}
\newcommand{\ea}{\end{eqnarray}}
\newcommand{\baa}{\begin{eqnarray*}}
\newcommand{\eaa}{\end{eqnarray*}}
\newcommand{\bb}{\begin{thebibliography}}
\newcommand{\eb}{\end{thebibliography}}

\newcommand{\lab}[1]{\label{#1}}
\newcommand{\re}[1]{(\ref{#1})}

\newcounter{my}
\newcommand{\he}%
   {\stepcounter{equation}\setcounter{my}%
   {\value{equation}}\setcounter{equation}0%
   }%
\newcommand{\she}%
   {\setcounter{equation}{\value{my}}%
    }%

\renewcommand\t{\tilde}

\newtheorem{theorem}{Theorem}[section]
\newtheorem{corollary}[theorem]{Corollary}
\newtheorem{lemma}[theorem]{Lemma}

\theoremstyle{definition}

\newtheorem{remark}[theorem]{Remark}

\numberwithin{equation}{section}

\DeclareRobustCommand{\erase}{\bgroup\markoverwith{\textcolor{red}{\rule[.5ex]{2pt}{0.8pt}}}\ULon}

\title{Eigenvalue equations for sieved polynomials 
or proving Askey right again}
\author{Luc Vinet and Alexei Zhedanov}

\address{IVADO and Centre de recherches math\'ematiques, Universit\'e de Montr\'eal, P.O. Box 6128, Centre-ville Station, Montr\'eal (Qu\'ebec), H3C 3J7}

\address{Department of Mathematics, School of Information, Renmin University of China, Beijing 100872,CHINA}

\begin{document}

\begin{abstract}
The sieved Jacobi polynomials have been introduced by Askey. These can be obtained from conveniently taking $q$ to be a root of unity in the Askey-Wilson polynomials. The question of determining if they are eigenfunctions of some operator has been lingering for a long time. Askey impressed on us his conviction that it had an affirmative answer. It is shown that he was right and that this operator is of Dunkl type with cyclic reflections corresponding to the powers of $q$.
\end{abstract}

\maketitle

\section{Introduction}
Dick Askey was kindly enthusiastic about colleagues' work and generous with comments and suggestions. This paper is rooted in such a manifestation of Askey's remarkable benevolence and insights regarding our past work \cite{tsujimoto2012dunkl}, \cite{vinet2012limit} on the bispectrality of Bannai-Ito polynomials \cite{bannai2021algebraic}. Indeed, every time we would present this work in his presence, Dick would express his appreciation but would not miss the occasion to say that he felt the results could be extended beyond $-1$-polynomials to a class of polynomials referred to as sieved that he had introduced with colleagues \cite{al1984sieved} and on which Mourad Ismail and his collaborators have written abundantly (see for instance \cite{charris1987sieved}, \cite{bustoz1990sieved}, \cite{ismail1992sieved}). The issue was of casting the sieved polynomials as eigenfunctions of some operator acting on the variable as we had managed to do for the Bannai-Ito ones. We would typically reply that we were doubtful this was possible only to find recently from studying one family of these sieved polynomials \cite{vinet2024cmv}, \cite{vinet2025bispectrality}, that Dick was right at least in that case. The CMV formalism \cite{cantero2003five} proved to be a guiding light in the matter. Indeed the desired result could be found by considering first the sieved Jacobi orthogonal polynomials on the unit circle (OPUC), passing to the Laurent polynomials associated to these OPUC, determining the operator in the circle variable $z$ that is diagonal on these polynomials to arrive at the eigenvalue equation of the sieved Jacobi orthogonal polynomials on the real line using their expression in terms of the Laurent polynomials through the Szeg\H{o} map and the differential equation of the latter.

This is the story that we will tell here. We shall first set the stage in Section 2 with some background on sieved polynomials and Bannai-Ito polynomials recalling in particular that the latter are eigenfunctions of a Dunkl shift operator involving reflections (by definition) to give context to the challenge that Askey repeatedly presented to us. In the rest of the paper we shall explain how we could finally confirm Askey's intuition that sieved polynomials are eigenfunctions of differential/difference operators in the Jacobi case. Because they will be used throughout, the general definition of OPUC and of the associated CMV Laurent polynomials as well as the expressions of the corresponding OPRL introduced by Szeg\H{o} are recalled in Section 3. The sieved Jacobi OPUC and OPRL are then presented in Section 4. The operator of which the CMV Laurent polynomials associated to the sieved Jacobi OPUC are eigenfunctions is found in Section 5 and this result will be seen in Section 6 to lead to the differential/difference eigenvalue equation of the sieved Jacobi OPRL. The special case of the sieved ultraspherical polynomials is discussed in Section 7 and conclusions follow.

\section{Background}
This section offers historical remarks regarding the question addressed in this paper and the central role that Dick Askey played in the matter.
\subsection{Sieved polynomials}
The sieved orthogonal were first introduced in a paper by Al-Salam, Allaway and Askey \cite{al1984sieved} entitled \textit{Sieved ultraspherical polynomials}. They considered the $q$-ultraspherical polynomials and observed that by letting $q$ be a $k$-root of unity ($q\rightarrow exp{\frac{2i\pi}{k}}, k>2$), new families of orthogonal polynomials result whose recurrence relations are organized in $k$ blocks. They have called this process sieving as the ensuing relations are those of a certain (simple) family of polynomials except for a certain value of the degree, mod $k$, when the recurrence relations become different as if they had been sifted. According to how the limit that sets $q$ equal to a root of unity is taken, this sieving can lead to different types of polynomials and families of first and second kind have been distinguished. 

Subsequently in the survey paper \cite{askey1984orthogonal} published in a Conference Proceedings, Askey applied the sieving procedure to the Askey-Wilson polynomials to obtain two families that he called the sieved Jacobi polynomials of first and second kind. (See also \cite{badkov1987systems}.) The detailed derivation of the recurrence coefficients of these sieved polynomials can be found in \cite{aldana1998block} as well as in DeSesa's thesis \cite{desesa1992sieved}. Additional notes of interest on the discovery of the sieved polynomials are also given in these two references. 

As sieved OPUC will be called upon in our endeavour to find the eigenvalue equation of the sieved Jacobi polynomials, it is relevant to point out that Ismail and Li introduced in \cite{ismail1992sieved} the unit circle version of the sieved ultraspherical polynomials (see also \cite{marcellan1991orthogonal}) and using the Szeg\H{o} map showed their correspondence with the sieved polynomials of the first and second kind orthogonal on the interval $[-1,1]$ of the real line.

Already in his first paper with Al-Salam and Allaway \cite{al1984sieved} on the topic, Askey raised the question of the eigenvalue equation of the sieved polynomials writing in the second paragraph of the conclusion of this article: `` A potentially very important result would be the second order differential these polynomials satisfy. See Atkinson and Everitt \cite{atkinson1981orthogonal} for a proof of the existence of these equations." Together with his co-authors Askey was imagining that it could be possible to derive these equations from the $q$-divided difference of the Askey-Wilson polynomials. It turns out however that this is not possible. A differential equation could nevertheless be found and was reported in \cite{bustoz1990sieved}; it is important to stress that the equation given in this paper is not an eigenvalue one. The problem of finding the operator of which the sieved polynomials are eigenfunctions and hence of demonstrating their bispectrality hence remained.

\subsection{Bannai-Ito polynomials}

The Bannai-Ito (BI) polynomials bear the names of the two distinguished colleagues who introduced them in their classification of the $P$- and $Q$-polynomial association schemes \cite{bannai2021algebraic}. This classification makes use of an important theorem of Douglas Leonard \cite{leonard1982orthogonal}. Bannai and Ito showed that these BI polynomials are obtained as a $q\rightarrow -1$ limit of the $q$-Racah polynomials. (We see here a special root of unity lurking.) They determined their recurrence relations and gave explicit expression of these polynomials as linear combinations of two hyperfunctions $_{4}F_3(1)$. However, thirty years after their introduction, their full characterization and the clear determination of their bispectral properties were still lacking. In his development of the theory of Leonard pairs as a linear algebra tool to describe $P$- and $Q$- association schemes, Paul Terwilliger systematically included the case corresponding to these polynomials in his studies (see, for instance, \cite{terwilliger2004two}, \cite{terwilliger2005two}) and, moreover, advocated that the special function experts pay more attention to these Bannai-Ito polynomials. 

We came to the examination of the Bannai-Ito polynomials as follows. Around 2010, we started the exploration of what we have called $-1$ orthogonal polynomials by looking at the $q \rightarrow -1$ limit first, of the little $q$-Jacobi polynomials and then of the big $q$-Jacobi polynomials \cite{koekoek2010hypergeometric}. The resulting polynomials were designated by replacing $q$ by $-1$ in the names of their parents. We found that these little \cite{vinet2011missing} and big $-1$ polynomials \cite{vinet2012limit}, with the former a special case of the latter, are eigenfunctions of a first order differential operators of Dunkl type that involve reflection operators. Furthermore, we showed in \cite{vinet2011bochner} that the big $-1$ Jacobi polynomials and their specializations exhaust (under reasonable conditions) the list of orthogonal polynomials that are eigenfunctions of first-order Dunkl operators $L$ of the form:
\begin{equation}
L=F_0(x)+F_1(x)R_{-}+G_0(x)\partial_x+G_1(x)\partial_xR_{-},
\end{equation}
where $R_{-}$ is the reflection operator that acts according to $R_{-}f(x)=f(-x)$ on functions $f$ of $x$. The functions $F_0, F_1, G_0, G_1$ that are initially taken to be arbitrary reduce in the solution of the problem to two real-valued functions depending on three parameters (see \cite{vinet2011bochner} for the details).

Now, knowing that the big $-1$ Jacobi polynomials can be obtained, as per Figure 1 below, from the Bannai-Ito polynomials by letting the truncation parameter $N$ go to $\infty$ and carefully adjusting the behavior of certain parameters under this limit (see \cite{vinet2012limit} or \cite{tsujimoto2012dunkl}), it was natural to think that the Bannai-Ito could be eigenfunctions of a shift operator of Dunkl type. 
\begin{figure}
\scalebox{0.85}{
\begin{tikzpicture}[
  node distance=1.7cm and 2.4cm,
  box/.style={draw, minimum width=2cm, minimum height=0.6cm, align=center},
  arrow/.style={-{Latex}, thick}
]

\node[box] (A) {q-Racah};
\node[box, right=of A] (B) {Bannai-Ito};
\node[box, below=of A] (C) {big \( q \)-Jacobi};
\node[box, below=of B] (D) {big \(-1\) Jacobi};

\draw[arrow] (A) -- (B) node[midway, above] {\( q \to -1 \)};
\draw[arrow] (C) -- (D) node[midway, above] {\( q \to -1 \)};

\draw[arrow] (A) -- (C) node[midway, left] {\( N \to \infty \)};
\draw[arrow] (B) -- (D) node[midway, right] {\( N \to \infty \)};

\end{tikzpicture}
}
\caption{ The big $q$-Jacobi polynomials which can be obtained from the $q$-Racah polynomials in the limit $N \rightarrow \infty$ \cite{koekoek2010hypergeometric}, yield the big $-1$ Jacobi polynomials when $q \rightarrow -1$. These big $-1$ Jacobi polynomials can also be obtained by taking the $q \rightarrow -1$ limit of the $q$-Racah polynomials to find the terminating Bannai-Ito polynomials and then letting $N \rightarrow \infty$. This implies that the diagram above is commutative.}
\end{figure}
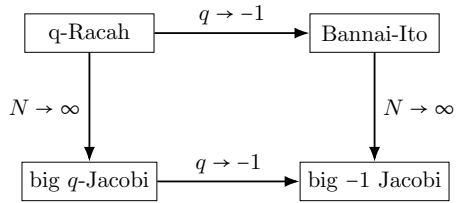


Let $T^+$ be the shift operator that has the following action on functions $f$ of $x$: $T^+f(x)=f(x+1)$. We have shown in \cite{tsujimoto2012dunkl} that the most general operator that is of first order in $T^+$, contains $R$ and, preserves the space of polynomials in $x$ of a fixed degree $n$ is of the form:
\begin{equation}
    L=F(x)(I-R)+G(x)(T^+R-I),
\end{equation}
where $I$ is the identity and $F(x)$ and $G(x)$ are given by
\begin{equation}
    F(x)=\frac{(x-\rho_1)(x-\rho_2)}{2x}, \qquad G(x)=\frac{(x-r_1+\frac{1}{2})(x-r_2+\frac{1}{2})}{2x+1},
\end{equation}
with $\rho _1, \rho _2, r_1, r_2$ four arbitrary real parameters. The eigenvalues $\sigma _n$ of these operators are:
\begin{equation} \label{eigenvaluebi}
    \sigma _n = \left\{
    \begin{aligned}
    &\;\frac{n}{2} \quad \text{if $n$ is even}\\
    &\;\rho _1 + \rho _2 + r_1+r_2-\frac{(n+1)}{2} \quad \text{if $n$ is odd}.
    \end{aligned}
    \right.
\end{equation}
Confirming our intuition, we then found in \cite{tsujimoto2012dunkl} that the four-parameter polynomial eigenfunctions of this operator $L$ with eigenvalues $\sigma _n$ given by \eqref{eigenvaluebi} are the general untruncated Bannai-Ito polynomials. Introducing the Bannai-Ito algebra realized with $L$ and multiplication by $x$ as generators, it was seen \cite{tsujimoto2012dunkl} that the whole characterization of the Bannai-Ito polynomials unfolds from this identification of their eigenvalue equation. 

It is worth stressing that the untruncated Bannai-Ito polynomials are not positive-definite. A finite set of $N+1$ elements that are positive-definite can however be obtained through a standard truncation procedure that yields a family of polynomials that are orthogonal on the grid made of the zeros $x_s,\; s=0, \dots N$ of the monic Bannai-Ito polynomial $P_{N+1}(x)$ of degree $N+1$. This Bannai-Ito grid satisfies the linear relation $x_{s+1}+x_{s-1}+2x_s+1=0$. What is remarkable is that the Dunkl shift operator $L$ becomes a three-diagonal matrix when acting on the finite ordered set $f(x_s), \;s=0,\dots, N$. This explains why the Leonard pair associated to the truncated Bannai-Ito polynomials could be derived from the equation $LP_n(x)=\lambda _n P_n(x)$ obeyed by the (untruncated) Bannai-Ito polynomials.

These striking results regarding the difference relation (or differential) eigenvalue equation for polynomials that are obtained as limits of a the real root $q=-1$ of the Askey-Wilson polynomials (or of the Big $q$- Jacobi) polynomials explains why Askey kept insisting that we investigate the question for other, more generic, roots of unity so as to obtain the operators that are diagonal on the sieved polynomials he had introduced.

We already announced that OPUC will play a central role in the resolution of Askey's challenge. Let us mention here that in works with Derevyagin, we have explained how the little $-1$ and big $-1$ Jacobi polynomials can be obtained from the Jacobi polynomials on the unit circle \cite{derevyagin2012cmv} and as well how the Bannai-Ito polynomials are related to OPUC \cite{derevyagin2014bannai}.

\section{OPUC and Laurent polynomials} \label{sect:2}
\setcounter{equation}{0}
To begin, we recall basic facts about orthogonal polynomials on the unit circle \cite{simon2005orthogonal}, \cite{simon2005opuc} and their maps to orthogonal polynomials on the real line.
\subsection{OPUC}
The OPUC $\Phi_n(z)=z^n + O(z^{n-1})$ are monic polynomials defined by the recurrence relation 
\be
\Phi_{n+1}(z) = z \Phi_n(z) -  a_n \Phi_n^*(z), \quad \Phi_0(z) =1 \lab{Sz_rec} \ee 
where
\be
\Phi_n^*(z) = z^n \Phi_n(z^{-1}).
\lab{Phi*} \ee
The Verblunsky parameters $a_n$ satisfy the condition
\be
|a_n| <1 . \lab{a<1} \ee
In what follows we shall only consider polynomials $\Phi_n(z)$  with real coefficients $a_n$. These $\Phi_n(z)$ are then orthogonal on the unit circle
\be
\int_{-\pi}^{\pi} \Phi_n(\exp(i \theta))  \Phi_m(\exp(-i \theta)) \rho(\theta) d \theta = h_n \delta_{nm}
\lab{ort_real} \ee
with $\rho$ a positive and symmetric weight function:
\be
\rho(-\theta) = \rho(\theta).
\lab{rho_sym} \ee
The normalization constants
\be
h_0 =1, \quad h_n = (1-a_0^2)(1-a_1^2) \dots (1-a_{n-1}^2) , \quad  n=1,2,\dots
\lab{h_n} \ee
are all positive  $h_n>0, \: n=0,1,\dots$ because of condition \re{a<1}.

\subsection{Laurent polynomials}
\subsubsection{The CMV Laurent polynomials}

The Laurent polynomials
\be
\psi_{2n}(z) = z^n \Phi_{2n}(1/z), \quad \psi_{2n+1}(z) = z^{-n} \Phi_{2n+1}(z)
\lab{psi_def} \ee
that satisfy similar orthogonality relation, namely
\be
\int_{-\pi}^{\pi} \psi_n(\exp(i \theta))  \psi_m(\exp(-i \theta)) \rho(\theta) d \theta = h_n \delta_{nm}.
\lab{ort_psi} \ee
will prove central in our approach to identify the sought after eigenvalue problems in the CMV-formalism .

\subsubsection{The Szeg\H{o} Laurent polynomials}

Consider in addition the set of Laurent polynomials
\be
P_0 =1, \quad P_n = z^{1-n} \Phi_{2n-1}(z) + z^{n-1} \Phi_{2n-1}(1/z) = \psi_{2n-1}(z) + \psi_{2n-1}\left(1/z \right), \; n=1,2,\dots 
\lab{P_Phi} \ee 
These polynomials are symmetric with respect to the operation $z \to 1/z$. Hence, one can conclude that $P_n$ are monic polynomials of the argument $x(z) =z+1/z$:
\be
P_n = P_n(x(z)) = P_n(z+1/z) = (z+1/z)^n + O\left( \left(z+1/z \right)^{n-1} \right).
\lab{P_x} \ee
They indeed satisfy the recurrence relation
\be
P_{n+1}(x) + b_n P_{n}(x) + u_n P_{n-1}(x) = xP_n(x) 
\lab{rec_P_Sz} \ee
where the recurrence coefficients are
\ba
&&u_n = (1+a_{2n-1}) (1-a_{2n-3}) (1-a_{2n-2}^2), \nonumber \\
&&b_n = a_{2n}(1-a_{2n-1})-a_{2n-2}(1+a_{2n-1}).
\lab{xi_Sz} \ea
The condition $a_{-1}=-1$ is assumed in these formulas. 

The ``companion" polynomials $Q_n$ are defined through the formula
\be
Q_n(z) = \frac{z^{-n} \Phi_{2n+1}(z) - z^{n} \Phi_{2n+1}(1/z)}{z-z^{-1}} = \frac{\psi_{2n+1}(z) - \psi_{2n+1}(1/z)}{z-z^{-1}}.
\lab{Q_Sz} \ee
These polynomials $Q_n$ are again monic and of degree $n$ in the same  argument $x(z)$ and they satisfy the recurrence relation 
\be
Q_{n+1}(x) + \t b_n Q_{n}(x) + \t u_n Q_{n-1}(x) = xQ_n(x) 
\lab{rec_Q_Sz} \ee
with
\ba
&&\t u_n = (1+a_{2n-1})(1-a_{2n+1})(1-a_{2n}^2), \nonumber \\ &&\t b_n = a_{2n}(1-a_{2n+1}) -a_{2n+2}(1+a_{2n+1}).
\lab{xi_Sz_Q} \ea
The polynomials $P_n$ and $Q_n$ were introduced by Szeg\H{o} \cite{szeg1939orthogonal}. They are orthogonal on the interval $[-2,2]$
\be
\int_{-2}^{2} P_n(x) P_m(x) w(x) d x = 0, \; \mbox{if} \; m\ne n,
\lab{ort_P} \ee   
and 
\be
\int_{-2}^{2} Q_n(x) Q_m(x) w(x)(4-x^2) d x = 0, \; \mbox{if} \; m\ne n,
\lab{ort_Q} \ee 
where
\be
w(x) = \frac{\rho(\theta)}{\sqrt{4-x^2}}
\lab{w_rho} \ee
and $\rho(\theta)$ is understood to be a function of $x$. Indeed, in light of \re{rho_sym}, $\rho(\theta)$ is symmetric and hence depends only on $x = 2 \cos \theta$.

The polynomials $Q_n(x(z))$ can be obtained from $P_n(x(z))$ by the double Christoffel transform at the two spectral points $x= \pm 2$. This leads to the following formula expressing $Q_n(x(z))$ in terms of $P_n(x(z))$:
\ba
&&\left(z-z^{-1}\right)^2 Q_{n-1}(x(z)) = P_{n+1}(x(z)) + (a_{2n}+a_{2n-2})(1-a_{2n-1})P_n(x(z)) \nonumber \\ &&-(1-a_{2n-1})(1-a_{2n-3})\left(1-a_{2n-2}^2 \right)P_{n-1}(x(z)). \lab{Q->P} \ea
Using the recurrence relation \re{rec_P_Sz}, one can present \re{Q->P} in the equivalent form
\ba
&&\left(z-z^{-1}\right)^2 Q_{n-1}(x(z)) = \nonumber \\
&&\left( z+z^{-1} + 2a_{2n-2} \right) P_n(x(z)) - 2\left(1-a_{2n-3} \right) \left( 1-a^2_{2n-2}\right)P_{n-1}(x(z)). \lab{Q-->P} \ea
The mapping from $Q_n(x)$ to $P_n(x)$ is given by the Geronimus transformation
\ba
&&P_n(x) = Q_n(x) - (1+a_{2n-1})(a_{2n}+a_{2n-2})Q_{n-1}(x) -  \nonumber \\ 
&&(1+a_{2n-1})(1+a_{2n-3})\left( 1-a_{2n-2}^2 \right)Q_{n-2}(x).
\lab{PQG} \ea
The inverse map from the polynomials $P_n(x), Q_n(x)$ to $\psi_n(z)$ reads
\ba
&&\psi_{2n-1}(z) = \frac{1}{2} \left[ P_n(x(z)) +\left( z-z^{-1} \right) Q_{n-1}(x(z)) \right] ,  \nonumber \\
&&\psi_{2n}(z) = \frac{1}{2} \left[ (1-a_{2n-1}) P_n(x(z)) -(1+a_{2n-1})\left( z-z^{-1} \right) Q_{n-1}(x(z))\right].  
\lab{psi_PQ}
\ea
From these formulas we have another useful expression of the polynomials $P_n(x)$ and $Q_n(x)$ in terms of the CMV ones:
\ba
&&P_n(x(z)) = \psi_{2n}(z) +\left( 1+a_{2n-1}\right) \psi_{2n-1}(z), \nonumber \\
&&\left(z-z^{-1}\right)Q_{n-1}(x(z)) = -\psi_{2n}(z) +\left( 1- a_{2n-1}\right) \psi_{2n-1}(z).
\lab{PQ_psi} \ea

\section{Sieved Jacobi OPUC and OPRL} \label{sect:3}
\setcounter{equation}{0}
This section focuses on the sieved Jacobi OPUC and OPRL that we present next and whose eigenvalue equations we shall be looking for in the next two sections.
\subsection{The sieved Jacobi OPUC}
Consider the coefficients $a_n, n=0,1, \dots,$ defined as follows 
\be
a_n = - \frac{\alpha + 1/2+ (-1)^{n+1} \left(  \beta +1/2 \right) }{n+\alpha+\beta+2}, \quad n=0,1,2,\dots
\lab{an_sol} \ee
with $\alpha, \beta$ two real parameters.
Let $N$ be a fixed natural number $N = 1,2,\dots$  
The sieved Jacobi OPUC $\Phi_n(z;N)$ are determined by the
Verblunsky parameters $a_n(N)$ that are given in terms of $a_n$ through the relations \cite{ismail1992sieved}:
\be \label{sieving}
a_n(N) = \left\{ a_{k-1}  \quad \mbox{if} \; n=Nk-1  \atop  0 \quad \quad \mbox{otherwise} .\right.
\ee
When $N=1$, we have $a_n(1)=a_n$ and the polynomials $\Phi_n(z;1)$ correspond to the (unsieved) Jacobi OPUC. Equation \eqref{sieving} thus accomplish the sieving of these polynomials. 

The sieved Jacobi OPUC are orthogonal on the unit circle
\be
\int_{0}^{2\pi} \rho(\theta;N) \Phi_n\left(e^{i \theta};N\right) \Phi_m \left(e^{-i \theta};N\right) d \theta = 0, \; n \ne m,
\lab{ort_sJOPUC} \ee
with respect to the positive weight function
\be
\rho(\theta;N) = \left( 1-\cos N\theta \right)^{\alpha+1/2}  \left( 1+ N\cos \theta \right)^{\beta+1/2}.
\lab{rho_N} \ee

It has been shown in \cite{ismail1992sieved} that
\be
\Phi_n(z;N) = z^j \Phi_k(z^N;1), \quad \mbox{where} \quad n=Nk+j, \quad j=0,1,\dots, N-1.
\lab{Phi(N)} \ee 
We shall use the corresponding Laurent CMV-polynomials that are defined according to \eqref{psi_def} 
\be
\psi_{2n}(z;N) = z^n \Phi_{2n}\left(z^{-1};N \right), \;  \psi_{2n+1}(z;N) = z^{-n} \Phi_{2n+1}\left(z;N \right).
\lab{sieved_psi} \ee
One could expect relations between $\psi_n(z;N)$ and $\psi_n(z;1)$ that are similar to formula  \re{Phi(N)}. They indeed exist, but depend on the parity of the numbers $N,n,j$.\\
If $n$ is {\it even}, one has 
\be
\psi_n(z;N) = \left\{  z^{-j/2} \psi_k \left(z^N;1 \right),  \;  j \; \mbox{even}   \atop   z^{(j+1)/2} \psi_k \left(z^{-N};1 \right), \;  j \; \mbox{odd}.\right.
\lab{psi-psi_n_even} \ee
If $n$ is {\it odd} and $N$ is {\it even}, one finds
\be
\psi_n(z;N) = \left\{  z^{(N-j)/2} \psi_k \left(z^{-N};1 \right),  \;  j \; \mbox{even}   \atop   z^{(-N+j+1)/2} \psi_k \left(z^{N};1 \right), \;  j \; \mbox{odd}.\right.
\lab{psi-psi_n_odd_N_even} \ee
Finally, if $n$ is {\it odd} and $N$ is {\it odd}, one gets
\be
\psi_n(z;N) = \left\{  z^{(-N+j+1)/2} \psi_k \left(z^{N};1 \right),  \;  j \; \mbox{even}   \atop   z^{(N-j)/2} \psi_k \left(z^{-N};1 \right), \;  j \; \mbox{odd}.\right. ,
\lab{psi-psi_n_odd_N_odd} 
\ee
where
\begin{equation}
    n=Nk+j, \quad j=0,1,\dots, N-1.
\end{equation}
\subsection{The sieved Jacobi OPRL}
The ``sieved" polynomials $P_n(z;N)\:$ and $ \: Q_n(z;N)$ that are orthogonal  on the real line are defined as follows in terms of the Laurent CMV-polynomials $\psi_{2n}(z;N)$ and $\psi_{2n+1}(z;N)$ according to the general Szeg\H{o} framework given in Section 3:
\ba
&&P_n(x\left(z);N \right) = \psi_{2n}(z;N) + \left(1+a_{2n-1}(N)\right) \psi_{2n-1}(z;N) \lab{P_sieved}\\
&& \left(z-z^{-1} \right)Q_{n-1}(x\left(z);N \right) = -\psi_{2n}(z;N) + \left(1-a_{2n-1}(N)\right) \psi_{2n-1}(z;N).
\lab{Q_sieved} \ea 
The polynomials $P_n (x(z);N)$ and $Q_n (x(z);N)$ are respectively called the sieved
Jacobi polynomials of the first and second kind. Like the Jacobi PRL, in view of \eqref{an_sol}, \eqref{sieving}, \eqref{sieved_psi}, \eqref{P_sieved} and \eqref{Q_sieved}, the polynomials  $P_n (x(z);N)$ and $Q_n (x(z);N)$ depend on the two parameters $\alpha$ and $\beta$. An important special case arises when these parameters are taken to be equal, that is when $\alpha = \beta$; this leads to the sieved ultraspherical polynomials on the real line that were first introduced in \cite{al1984sieved}. They will be discussed in Section 7. 

\section{The eigenvalue equation of the CMV Laurent polynomials associated to the sieved Jacobi OPUC}
As explained in the first two sections of this paper, in spite of Askey's conviction that it exists and of other arguments, the differential/difference eigenvalue equation of the sieved Jacobi polynomials has so far remained elusive to the best of our knowledge. We have managed to ``crack" this problem by splitting it in two parts focusing first on the unit circle framework and the eigenvalue equation of the associated CMV Laurent polynomials as a step towards obtaining the eigenvalue equation of the sieved OPRL. This section will bear on the first part and the eigenvalue equation obeyed by the Laurent polynomials
$\psi_n(z;N)$.

We shall use the reflection operators $R_j$ and $zR_j$ with 
\be
R_j f(z)= f\left( q^j/z\right), \; j=0,1,2,\dots, N-1,
\lab{R_i_def} \ee
and where $q$ is a primitive $N$-th root of unity. We shall take
\be
q=\exp\left(\frac{2 \pi i}{N} \right).
\lab{q_def} \ee

We first need the following lemma that provides the differential/difference eigenvalue equation of the special Laurent polynomial $\psi_n(z;1)$ with $N=1$, that corresponds in fact to the Jacobi OPUC \cite{vinet2024cmv}.

\begin{lemma} \label{Lemma 1}
Let  $\mathcal{K}$ be the following operator
   \be
\mathcal{K}= z \partial_z  + \frac{z\left((\alpha+\beta+1)z +\alpha-\beta \right)}{1-z^2} \left(R_--\mathcal{I} \right).
\lab{KJop} \ee
The CMV Laurent polynomials $\psi_n(z;1)$  satisfy the eigenvalue equation
\be
\mathcal{K} \psi_n(z;1) = \mu_n \psi_n(z;1),
\lab{EIG_JOPUC} \ee
where 
\be
\mu_n = \left\{ -n/2,  \; n \; \mbox{even} \atop  (n+1)/2+\alpha+\beta+1 , \; n \; \mbox{odd} . \right.
\lab{l_JOPUC} \ee

\end{lemma}

\begin{proof}
    It was already shown by Szeg\H{o} \cite{szeg1939orthogonal} that the pair of PRL $P_n(x(z);1), Q_n(x(z);1)$ corresponding to the Jacobi OPUC $\Phi(z;1)$  coincides with the ordinary Jacobi polynomials $P_n^{(\alpha, \beta)}(x/2)$ and $P_n^{(\alpha+1, \beta+1)}(x/2)$ which are orthogonal on the interval $[-2,2]$ with the weight functions $(2-x)^{\alpha}(2+x)^{\beta}$ and $(2-x)^{\alpha+1}(2+x)^{\beta+1}$ respectively.

The identity
\be
z \partial_z P_n(x(z);1) = n (z-1/z) Q_{n-1}(x(z);1)
\lab{rel_PQ} \ee
follows from the well known relation between Jacobi polynomials with parameters $(\alpha,\beta)$ and $(\alpha+1,\beta+1)$ \cite{koekoek2010hypergeometric}.
Moreover, the differential equation that the Jacobi polynomials $P_n(x(z);1)$ satisfy \cite{koekoek2010hypergeometric} can be presented in the form
\be
z^2 \partial_z^2 P_n(x(z);1) + z\frac{(\alpha+\beta+2)z^2+2(\alpha-\beta)z +\alpha+\beta }{z^2-1} \partial_z P_n(x(z);1) = n(n+\alpha+\beta+1) P_n(x(z);1)
\lab{DEP} \ee

Now use \re{rel_PQ} to replace $Q_{n-1}(x(z);1)$ in \re{psi_PQ} with $\partial_z P_n(x(z);1)$ and note that $\psi_n(z;1)$ can then be presented as a linear combination of $P_n(x(z);1)$ and $\partial_z P_n(x(z);1)$. Applying the operator $\mathcal K$ to the formulas for $\psi_n(z;1)$  resulting from this substitution in \re{psi_PQ}, one obtains ${\mathcal K} \psi_n(z)$ as a linear combination of $P_n(x(z;1)), \partial_z P_n(x(z);1)$ and $\partial_z^2 P_n(x(z);1)$. Eliminating $\partial_z^2 P_n(x(z);1)$ with the help of \re{DEP}, one finally arrives at \re{EIG_JOPUC}.
\end{proof}

\begin{remark}
The operator $\mathcal K$ \re{KJop} is basically Cherednik's Dunkl differential operator that appears in the study \cite{koornwinder2011nonsymmetric} of the non-symmetric Jacobi polynomials which coincide with the CMV polynomials $\psi_n(z;1)$.
\end{remark}

We shall insert here a technical lemma that will be called upon in obtaining the eigenvalue equation of $\psi_n(z;N)$.
\begin{lemma}
Assume that $h$ is a non-negative integer smaller than $N \in \mathbb{N}$ and that $q^N=1$, the following identities hold:
 \begin{align}
   &z^2\sum_{l=0}^{N-1} \frac{q^{-\frac{lj}{2}}}{q^l-z^2} = N\frac{z^{2N-j}}{1-z^{2N}},  \label{sum1}   \\
     & z^2\sum_{l=0}^{N-1} \frac{(-1)^l q^{-\frac{lj}{2}}}{q^l-z^2} 
      =z^2\sum_{l=0}^{N-1} \frac{q^{(\frac{N}{2}-\frac{j}{2})l}}{q^l-z^2} 
    = N\frac{z^{N-j}}{1-z^{2N}}.  \label{sum2}
\end{align}
 
\end{lemma}
\begin{proof}
The proof proceeds \cite{vinet2025bispectrality} by considering the contour integral 
\begin{equation}
    \oint_{|w|>1} dw \frac{w^h}{(w^N-1)(w-z)}, \qquad h\in \mathbb{N},
\end{equation}
over a circle of radius larger than $1$ in the complex plane and applying Cauchy's residue theorem. 
Observing that the integrand has poles at the roots of unity $w_l=q^l, l=0,\dots N-1$ with $q=e^{\frac{2i\pi}{N}}$ and at $w=z$, and keeping in mind that 
\begin{equation}
    \lim_{w\rightarrow q^l}\;\frac{(w-q^l)}{\Pi_{k=0}^{N-1}(w-q^k)}=\frac{N}{q^l}, 
\end{equation}
one thus finds
\begin{equation}
    \sum_{l=0}^{N-1}\frac{q^{l(h+1)}}{q^l-z}=N\frac{z^h}{1-z^N}. \label{sumgen}
    \end{equation}
    Formulas \eqref{sum1} and \eqref{sum2} are then obtained from specializing \eqref{sumgen}.
\end{proof}

We are now ready to give the central result of this section which provides the differential/difference operator of Dunkl type that is diagonal on the Laurent polynomials $\psi_n(z;N)$ \cite{vinet2025bispectrality}.
\begin{theorem}\label{thm_5.4}

The CMV-Laurent polynomials $\psi_n(z;N)$ associated to the sieved Jacobi OPUC $\Phi_n(z;N)$ satisfy the eigenvalue equation
\be
L(N) \psi_n(z;N) = \lambda_n(N) \psi_n(z;N), 
\lab{eig_SOPUC} \ee
where the operator $L(N)$ is
\be
L(N) = z\partial_z + \sum_{k=0}^{N-1} A_k(z;N) \left( R_k - \mathcal{I}\right) 
\lab{L(N)} \ee
with
\be
A_k(z;N)= \sigma_k \frac{z^2}{q^k- z^2}, \; \sigma_k = \alpha+\beta+1  +(-1)^k (\alpha-\beta)
\lab{A_ev} \ee
for even $N$ and
\be
A_k(z;N) = \frac{(\alpha+\beta+1)z^2 + \rho_k(\alpha-\beta)z}{q^k- z^2}, \; \rho_k = \left\{  q^{k/2}, \; k \: \mbox{even}  \atop q^{(k-N)/2},  \; k \: \mbox{odd} \right. 
\lab{A_odd} \ee
for $N$ odd.
\\

\noindent The eigenvalues are
\be
\lambda_n(N) = \left\{ -n/2, \; n \: \mbox{even}  \atop (n+1)/2 + (\alpha+\beta+1)N, \; n \: \mbox{odd}.   \right.
\lab{lambda(N)} \ee
\end{theorem}
\begin{proof}    
The proof of this theorem is based on the eigenvalue equation \re{EIG_JOPUC} that the CMV-Laurent polynomials $\psi_n(z;1)$ obey and which is given in Lemma 1. It proceeds through a direct verification of \eqref{eig_SOPUC} as defined by \eqref{L(N)}-\eqref{lambda(N)}. In broad strokes, it goes as follows. One uses the relations between the sieved CMV Laurent polynomials $\psi_n(z;N)$ and the $N=1$ ones $\psi_k(z;1)$ provided by equations \eqref{psi-psi_n_even}-\eqref{psi-psi_n_odd_N_odd} to write the former in terms of the latter in the following fashion
\be
\psi_n(z;N) = z^{\nu} \psi_k\left(z^N;1 \right) \quad \mbox{or} \quad  \psi_n(z;N) = z^{\nu} \psi_k\left(z^{-N};1 \right),
\lab{psi_nu} \ee
where $\nu$ is an integer that depends on $n$, $N$ and $j$ with $n=Nk+j$. 
With $\psi_k'(z;1)$ the derivative of $\psi_k(z;1)$, we have from $\mathcal{K} \psi_k(z;1) = \mu_k \psi_k(z;1)$, \re{EIG_JOPUC}, that
\be
z^N \psi_k'\left( z^N;1\right) =   \left( G\left( z^N \right) + \mu_k \right)  \psi_k\left( z^N;1\right)  - G\left( z^N\right) \psi_k\left( z^{-N};1\right)
\lab{der_psiN} \ee
and a similar formula with $z^N$ replaced with $z^{-N}$. Substituting $\psi_n(z;N)$ as given by its expression of the type \eqref{psi_nu} into \re{eig_SOPUC} and 
replacing the derivative $\psi_k'\left( z^N;1 \right)$ by the relation \eqref{der_psiN} or as the case maybe, replacing $\psi_k'\left( z^{-N};1 \right)$ using the the relation with $z^N \rightarrow z^{-N}$, one will see that \eqref{eig_SOPUC} is transformed into an equation of the form
\be
F^{(1)}_n(z) \psi_n\left( z^N;1\right) + F^{(2)}_n(z) \psi_n\left( z^{-N};1\right) =0
\lab{FF} \ee
with the functions $F^{(1)}_n(z)$  and $F^{(2)}_n(z)$ containing $A(z^N),\; \mu_k$ and $\lambda_n(N)$. This is what needs to be proved at that point to confirm that Theorem \ref{thm_5.4} is true. As it turns out, it is possible to show that these functions $F^{(1)}_n(z)$  and $F^{(2)}_n(z)$ are identically zero separately to complete the proof.

There are many cases to check however. First $L(N)$ takes a different form depending on the parity of $N$. Second, the relations between the sieved CMV Laurent polynomials and the non-sieved ones split in numerous cases as per \eqref{psi-psi_n_even}-\eqref{psi-psi_n_odd_N_odd}. Let us indicate how the computations are carried out when $n$, $N$, (and $j$ then) are all even. In this situation, from 
\eqref{psi-psi_n_even} we have
\begin{equation}
    \psi_n(z;N)=z^{-\frac{j}{2}}\psi_k(z^N;1), \quad j=0, 2,\dots, N-2. \label{psi_special}
\end{equation}
From the Dunkl eigenvalue equation \eqref{KJop}-\eqref{l_JOPUC} for the CMV Laurent polynomials $\psi_k(z;1)$ we have:
\begin{equation}
    z^N\psi_k'(z^N;1) =- \big{[}\frac{(\alpha+\beta+1)z^{2N}+(\alpha-\beta)z^N}{1-z^{2N}}\big{]}\big{(}\psi_k(z^{-N};1)-\psi_k(z^N;1)\big{)}  -\frac{k}{2}\psi_k(z^N;1).\label{psiprime}
\end{equation}
Inserting \eqref{psi_special} in the proposed eigenvalue equation \eqref{eig_SOPUC} for the sieved Jacobi CMV OPUC, we find:
\begin{equation}
    (\frac{n}{2} - \frac{j}{2})z^{-\frac{j}{2}} \psi(Z^N;1) + Nz^{-\frac{j}{2}} Z^N\psi_k'(Z^N;1) + \sum_{l=0}^{N-1} A_l(z;N)(R_l-I)\big{(}z^{-\frac{j}{2}}\psi_k(z^N;1)\big{)}=0 \label{eig_int}
\end{equation}
with 
\begin{equation}
    A_l(z;N) = \big{[}(\alpha+\beta+1)+(-1^l)(\alpha-\beta)\big{]}, \quad l=0,\dots, N-1,
\end{equation}
for the particular case we are considering. Now
\begin{equation}
    (R_l-I) \big{(}z^{-\frac{j}{2}}\psi_k(z^N;1)\big{)}= q^{-\frac{lj}{2}}z^{\frac{j}{2}}\psi_k(z^{-N};1) - z^{-\frac{j}{2}}\psi_k(z^N;1).
\end{equation}
Putting all the pieces together in \eqref{eig_int}, i.e. replacing $z^N\psi_k'(z^N;1)$ by the r.h.s. of \eqref{psiprime} and using the sums \eqref{sum1} and \eqref{sum2}, one readily sees that the factors of $\psi_k(z^N;1)$ and of $\psi_k(z^{-N};1)$ vanish if one recalls that that $n=Nk+j$. All the other possible situations for $n$, $N$ and $j$ can be treated analogously to confirm that Theorem \ref{thm_5.4} holds.
\end{proof}

The operator $L(N)$ can be presented in the equivalent form using again the sums \eqref{sumgen}:
\be
L(N) = z\partial_z + \sum_{k=0}^{N-1} A_k(z;N)R_k +B(z) \mathcal{I},
\lab{L(N)_B} \ee
where 
\be
B(z) = -\sum_{k=0}^{N-1}A_k(z;N) = N \frac{(\alpha+\beta+1)z^{2N} + (\alpha-\beta)z^N}{z^{2N}-1}.
\lab{B_expl} \ee
Note that in contrast to the expressions of the coefficients $A_k(z;N)$, the one of $B(z)$ does not depend on the parity of $N$.
We also have the following result regarding the self-adjointness of $L(N)$:
\begin{theorem} \lab{Prop2}

The operator $L(N)$ is self adjoint on the unit circle with respect to the scalar product
\be
(f(z),g(z)) = \int_0^{2 \pi} f\left( e^{i \theta} \right) \bar g \left(e^{-i\theta}  \right) \rho(\theta;N) d \theta,
\lab{scal_N} \ee
where 
\be
\rho(\theta;N) = \left( 1-\cos N\theta \right)^{\alpha+1/2}  \left( 1+ \cos N\theta \right)^{\beta+1/2}. 
\lab{rho_N} \ee
\end{theorem}
\noindent In other words, this amounts to the following operator relation
\be
L(N)^{\dagger} \rho(\theta;N) = \rho(\theta;N) L(N),
\lab{self_L} \ee
with $L(N)^{\dagger}$ the adjoint of $L(N)$ and where it is assumed that $z=e^{i \theta}$. 

\begin{proof}
  Since $z\partial_z=-i\partial_\theta$, given \eqref{L(N)_B}, the adjoint of $L(N)$ reads:
  \begin{equation}
    L(N)^{\dagger}=-i\partial_\theta + \sum_{k=0}^{N-1} A_k(z,N)R_k+B^*(z) \label{Ldagger}
  \end{equation}
  taking into account that $R_k A_k(z)^*=A_k(z)R_k$ as is directly seen from the definition \eqref{R_i_def}, \eqref{q_def} of $R_k$ and the expressions \eqref{A_ev} and \eqref{A_odd} of $A_k(z)$. Now note that $\rho(\theta;N)=\rho(N\theta;1)$ and remark that the action of $R_k$ translates as follows on functions of $\theta$:
  \begin{equation}
      R_kf(\theta)=f(-\theta+\frac{2\pi k}{N}).
  \end{equation}
  To prove \eqref{self_L}, we first observe that $R_k\rho(\theta;N)=\rho(\theta;N)R_k$, namely that $\rho(\theta;N)$ is invariant under the inversions $R_k$. Indeed, using the periodicity of $\rho(\theta;1)$ and its reflection invariance \eqref{rho_sym}, we have:
  \begin{equation}
     R_k\rho(\theta;N)=\rho(-N\theta+2\pi k;1)R_k=\rho(-N\theta;1)R_k=\rho(N\theta;1)R_k= \rho(\theta;N)R_k.
  \end{equation}
Therefore showing that $L(N)^{\dagger} \rho(\theta;N) = \rho(\theta;N) L(N)$ reduces to checking that 
\begin{equation}
    -i\partial_\theta \;\ln \left(\rho(\theta;N)\right)=B(z)-B^*(z).
\end{equation}
One has,
\begin{equation}
    -i\partial_\theta \;\ln \left(\rho(\theta;N)\right) = -\frac{iN}{\sin N\theta}[(\alpha+\beta+1) \cos N\theta + (\alpha-\beta)]. \label{partial_ln}
\end{equation}
Writing \eqref{B_expl} in the form:
\begin{equation}
    B(z)=\frac{N}{z^N-\bar{z}^N}[(\alpha+\beta+1)z^N + (\alpha-\beta)], \quad \text{with} \quad \bar{z}=z^*=z^{-1},
\end{equation}
given that $\alpha$ and $\beta$ are real parameters, we readily find that
\begin{equation}
    B(z)-B^*(z)=\frac{N}{z^N-\bar{z}^N}[(\alpha+\beta+1)(z^N+\bar{z}^N)+ 2(\alpha-\beta)]
\end{equation}
which coincides when $z=e^{i\theta}$ with the expression \eqref{partial_ln} found for $-i\partial_\theta \;\ln \left(\rho(\theta;N) \right)$.
\end{proof}

For the ultraspherical case, i.e. when $\beta=\alpha$ the expressions for $A_k(z;N)$ and $B(z)$ become much simpler:
\be
A_k(z;N) = \frac{(2 \alpha+1)z^2}{q^k-z^2}
\lab{A_ultra} \ee
and 
\be
B(z) = N \frac{(2\alpha+1)z^{2N} }{z^{2N}-1}.
\lab{B_ultra} \ee


\section{Eigenvalue equations for the sieved Jacobi polynomials on the real line} \label{sect:6}
\setcounter{equation}{0}
In this section we present the eigenvalue equations for the sieved Jacobi polynomials of the first and second kind on the real line, thus finally fulfilling Askey's wish. As announced, these equations will be obtained with the help of the eigenvalue equation of the CMV Laurent polynomials associated to the sieved Jacobi OPUC given in Theorem 5.3.  

Introduce the operator
\be
H(N) = L^2(N) -N(\alpha+\beta+1)L(N).
\lab{H(N)} \ee
We have
\be
H(N) \psi_n(z;N) = \t \lambda_n(N) \psi_n(z;N)
\lab{H_psi} \ee
where
\be
\t \lambda_n(N) = \lambda^2_n(N) -N(\alpha+\beta+1)\lambda_n(N).
\lab{tlam} \ee
These eigenvalues satisfy the property
\be
\t \lambda_{2n} = \t \lambda_{2n-1} = \Lambda_n(N) 
\lab{tl_eo} \ee
where
\be
\Lambda_n(N) = n(n+N(\alpha+\beta+1)). 
\lab{Lambda_n} \ee
Apply the operator $H(N)$ to the polynomial $P_n(x(z);N)$. Using formulas \re{PQ_psi} and \re{tl_eo} we have:
\begin{align}
  H(N) P(x(z);N) &= H(N) \psi_{2n}(z;N) + \left(1+a_{2n-1}(N)\right) H(N) \psi_{2n-1}(z;N) \nonumber\\  
  &= \Lambda_n(N) P_n(x(z);N).
\lab{HPL}
\end{align}
Hence the polynomials $P_n(x(z);N)$ are eigenfunctions of the operator $H(N)$.

\noindent Similarly we have 
\be
\t H(N) Q_n(x(z);N)= \Lambda_{n+1}(N) Q_n(x(z);N),
\lab{HQL} 
\ee  
where the operator $\t H(N)$ is related to the operator $H(N)$ by the similarity transformation
\be
\t H(N) = \phi^{-1}(z) H(N) \phi(z)
\lab{tH} \ee
with 
\be
\phi(z) = z-z^{-1}.
\lab{phi(z)} \ee
We thus arrive at
\begin{theorem}
The sieved Jacobi polynomials $P_n(x(z);N)$ of the first kind satisfy the eigenvalue equation \re{HPL} while the sieved Jacobi polynomials of the second kind $Q_n(x(z);N)$ verify the eigenvalue equation \re{HQL}, where the operators $H(N)$ and $\t H(N)$ are defined by \re{H(N)} and \re{tH}.  
\end{theorem}
In light of definition \re{H(N)}, we can provide an explicit expression for the operator $H$. Indeed, from formula \re{L(N)_B} for the operator $L(N)$ we have 
\be
H(N) = z^2 \partial_z^2 + C(z;N) \partial_z + \sum_{k=0}^{N-1} D_k(z;N)\left(  R_k - \mathcal{I} \right) + \sum_{k=1}^{N-1} E_k(z;N) T_k
\lab{H(N)_sum} 
\ee
where the $N$ rotations $T_0, T_1, \dots, T_{N-1}$ are defined by
\be
T_k f(z) = f\left( q^k z \right),\quad k=0,\dots, N-1. \label{rot}
\ee 

\noindent The coefficients $D_k(z;N)$ and $E_k(z;N)$ read
\be
D_k(z;N) = A_k(z;N) \left( B(z) + B\left( q^k/z \right)  -N(\alpha+\beta+1) \right) + z A'_k(z;N),
\lab{E_k_gen} \ee 
\be
E_k(z;N) =  \sum_{i=0}^{N-1} A_i(z;N) A_{i+k}\left( q^i/z;N\right).
\lab{G_k_sum} \ee
In the sum \re{G_k_sum}, one can take into account the cyclic property of the coefficients $A_k(z)$:
\be
A_{k+N}(z;N) = A_k(z;N)
\lab{cycl_A} \ee
which follows straightforwardly from the explicit expressions \re{A_ev} and \re{A_odd}. 
The calculations then show that all the factors of the shift operators $T_k$ vanish:
\be
E_k(z;N) = 0, \quad k=1,\dots, N-1.
\lab{G=0} \ee
Moreover, it is seen that 
\be
 B(z) + B\left( q^k/z \right)  -N(\alpha+\beta+1) =0, \quad k=0,1, \dots, N-1
 \lab{BB=0} \ee
and that hence the coefficients in front of the reflection operators $R_k$ simplify to
\be  
D_k(z) = z A'_k(z;N).
\lab{E=zA} \ee
Finally, the term in front of the first derivative is found to be
\be
C(z) = z \left(1+N \left(\alpha +\beta +1\right)+\frac{2 N \left(\alpha +\beta +1+\left(\alpha -\beta \right) z^{N}\right)}{z^{2 N}-1}\right).
\lab{C(z)} \ee
Note last that on the space of the symmetric polynomials $P_n(z), Q_n(z)$ which only depend on the argument $x=z+1/z$, the reflection operators $R_k$ can be replaced by the rotation operators $T_{-k}$. Indeed, for any symmetric Laurent polynomial $f(z)=F(z+1/z)$, we have
\be
R_k f (z) = F\left(q^k/z + z q^{-k} \right) = f\left(zq^{-k}\right) = T_{-k} f(z).
\lab{R=T} \ee

 \noindent This is summed up as follows.
\begin{theorem}\lab{Prop4}
The eigenvalue operator $H(N)$ for sieved Jacobi polynomials $P_n(z;N)$ has two equivalent expressions: 

    \begin{equation}
    H(N)=\left\{
    \begin{aligned}
       & z^2 \partial_z^2 + C(z) \partial_z + z \sum_{k=0}^{N-1} A_k'(z;N) \left(R_{k} - \mathcal{I}  \right)  \\
&z^2 \partial_z^2 + C(z) \partial_z + z \sum_{k=1}^{N-1} A_k'(z;N) \left(T_{-k} - \mathcal{I}  \right) \lab{H_expl2} 
    \end{aligned}
\right.
\end{equation}       
where the formulas for $A_k(z;N)$ and $C(z)$ are given by \re{A_ev}-\re{A_odd} and \re{C(z)} and where $R_k$ and $T_k$ are the reflection and  rotation operators
\be
R_k f(z) = f\left(q^k/z\right), \quad   T_k f(z) = f\left( z q^k\right)
\lab{T_k_f}. \ee
\end{theorem}

\begin{remark}
The two operators of \eqref{H_expl2} providing expressions for $H(N)$ in Theorem \ref{Prop4} have the same action on the space of symmetric Laurent polynomials i.e. such that $f(z)=f(1/z)$ and are hence indistinguishable when acting on symmetric polynomials like $P_n(z)$ or $Q_n(z)$. Hovever, these operators are not equivalent on the space of arbitrary Laurent polynomials. 
\end{remark}
\begin{remark}
   It easily seen that the operator $H(N)$ is self-adjoint as it is obtained, according to formula \re{H(N)}, from the operator $L(N)$ that was shown to be self-adjoint in Proposition \ref{Prop2}. 
\end{remark}

\section{The sieved ultrasperical polynomials} \label{sect:8}
\setcounter{equation}{0}

When $\alpha=\beta$ and $N$ is arbitrary, the polynomials $P_n(x;N)$ and $Q_n(x;N)$ are respectively, the sieved ultraspherical polynomials of the first and second kind found in \cite{al1984sieved}. 

\subsection{First kind}
The sieved ultraspherical polynomials of the first kind obey the recurrence relation 
\be
P_{n+1}(x) + u_n P_{n-1}(x) = xP_n(x),
\lab{rec_gultra} \ee
with coefficients given by
\be
u_{nN} = \frac{2n}{2\alpha+2n+1}, \quad u_{nN+1} = \frac{4\alpha+2n+2}{2\alpha+2n+1}, \quad n=0,1,2\dots
\lab{u_1nd} 
\ee 
and
\be
u_n=1 \quad \text{for all other} \quad n. \nonumber
\ee
From the results of Section 6 we have:
\begin{corollary}

The sieved ultraspherical polynomials of the first kind $P_n(x;N)$ satisfy the eigenvalue equation
\be 
H(N) P_n\left(z+1/z; N \right) = n\left(n+ \left(2 \alpha+1 \right)N \right) P_n\left(z+1/z ; N \right), \; n=0,1,2,\dots,
\lab{EE_ultra} \ee
where the operator $H(N)$ has equivalently either one of the following two expressions
\begin{equation}
H(N) = \left\{
\begin{aligned}
    &z^2 \partial_z^2 + C(z) \partial_z + \sum_{k=0}^{N-1} B_k(z)\left(R_k - \mathcal{I}\right) \label{eq:L_ultra} \\
    &z^2 \partial_z^2 + C(z) \partial_z + \sum_{k=1}^{N-1} B_k(z)\left(T_{-k} - \mathcal{I}\right) 
\end{aligned}
\right.
\end{equation}

with 
\be
q=\exp\left(\frac{2\pi i}{N} \right), \quad T_k f(z) = \left(q^k z \right)
\lab{q_T} \ee
and
\be
B_k(z) = \frac{2(2 \alpha+1)q^k z^2}{\left( q^k -z^2\right)^2}, \quad C(z) = z \left[1+N(2 \alpha+1) + \frac{2N(2\alpha+1)}{z^{2N}-1} \right]. 
\lab{AC_ultra} 
\ee

\end{corollary}

\subsection{Second kind}

Coming to the sieved ultraspherical polynomials of the second kind, we see that those PRL have the recurrence coefficients 
\be
u_{nN} = \frac{2n}{2\alpha+2n+1}, \quad u_{nN-1} = \frac{4\alpha+2n+2}{2\alpha+2n+1}, \quad n=1,2, 3,  \dots
\lab{u_2nd} 
\ee 
and 
\be
u_n=1 \quad \text{for all other} \quad n. \nonumber
\ee
The operator $\t H(N)$ defining the specialization of the eigenvalue equation \eqref{HQL} for the polynomials $Q_n(z;N)$ can be obtained from formula \re{tH}. The spectrum $\Lambda_{n+1}(N)$ is provided by \eqref{Lambda_n} and it will be recalled that we are considering the particular case $\alpha=\beta$. For standardization purposes, the eigenvalue equation will be rewritten in the form $\hat{H}_n(N)Q_n(N)=\Xi_n(N)Q_n(N)$ with $\Xi_0(N)=0$. This will be achieved by subtracting the constant $\Lambda_1(N)$ from $\tilde{H}(N)$. Proceeding with this reformatting, it is readily seen that 
\begin{equation}
    \Xi_n(N)=\Lambda_{n+1}(N)-\Lambda_1(N)=n(n+N(2\alpha+1)+2),
\end{equation}
while in view of \eqref{tH}, $\hat{H}(N)$ is given by
\begin{equation}
    \hat{H}(N)=\left(z-z^{-1} \right)^{-1} H(N) \left(z-z^{-1} \right) - \left(N(2\alpha+1) +1 \right) \mathcal{I}.
\lab{tH_Q} 
\end{equation}
After straightforward computations to perform the conjugation and carry out simplifications, one arrives at the following result.

\begin{corollary}

The sieved ultraspherical polynomials of the second kind $Q_n(x;N)$ satisfy the eigenvalue equation
\be 
\hat{H}(N) Q_n\left(z+1/z; N \right) = n\left(n+ \left(2 \alpha+1 \right)N + 2\right) Q_n\left(z+1/z ; N \right), \; n=0,1,2,\dots,
\lab{EQ_ultra} \ee
where $\hat{H}(N)$ takes equivalently one of the following two forms
\begin{equation}
\hat{H}(N)=
    \left\{
\begin{aligned}
    &z^2 \partial_z^2 + \hat{C}(z) \partial_z + \sum_{k=0}^{N-1} \hat{B}_k(z)\left(R_{k} - \mathcal{I} \right)  
\lab{HQ_ultra}\\
&z^2 \partial_z^2 + \hat{C}(z) \partial_z + \sum_{k=1}^{N-1} \hat{B}_k(z)\left(T_{-k} - \mathcal{I} \right),  
\end{aligned}
\right.
\end{equation}
and where 
\be
\hat{B}_k(z) = \frac{q^{2k}-z^2}{q^k\left(z^2-1 \right)} B_k(z), \quad \hat{C}(z) = C(z) +\frac{2z(z^2+1)}{z^2-1}, 
\lab{BC_Q} \ee
with $B_k(z)$ and $C(z)$ given by \re{AC_ultra}. 
\end{corollary}

\section{Conclusion}

In the end, Dick was therefore right: the sieved Jacobi polynomials (and their special ultraspherical case) are eigenfunctions of a differential operator of Dunkl type. To be noted is the fact that this operator is expressed in the variable $z$, an indication it was obtained from the one diagonalizing the Laurent polynomials associated to the sieved Jacobi OPUC. This therefore means that both these Laurent polynomials and the sieved Jacobi OPRL are bispectral. In the case of the OPUC, one speaks of CMV bispectrality as there are a number of recurrence relations in the form of pentadiagonal CMV relations, see \cite{vinet2025bispectrality}. A problem of great interest is to determine the algebras that encode these bispectralities for both the sieved Jacobi OPRL and OPUC. In the case of trivial sieving ($N=1)$, when one really considers the Jacobi OPUC, this algebra was identified in \cite{vinet2024cmv} and called the circle Jacobi algebra. It was seen that its representations provide the corresponding Verblunsky parameters. The extension to arbitrary $N$ is bound to be enlightening as this could lead to a representation theoretic description of the sieved polynomials. One should of course pursue the exploration of the eigenvalue equations of other families of sieved polynomials such as the Pollaczek one \cite{charris1987sieved}. All this adds another example illustrating that Askey's life and work remain a constant source of inspiration.

\section*{Acknowledgments}
We wish to reiterate our posthumous gratefulness to Dick Askey for inviting us with kindness and insistence to look for the eigenvalue equation of sieved polynomials. We also thank Mourad Ismail for inspiring discussions. LV is funded in part through a discovery grant of the Natural Sciences and Engineering Research Council (NSERC) of Canada. 

\bibliographystyle{unsrt} 
\bibliography{ref_sievAsk.bib}

\end{document}